\documentclass[11pt, twoside]{amsart}

\usepackage{color}

\usepackage{graphicx}
\usepackage{bookmark}

\usepackage{hyperref}
\hypersetup{bookmarksdepth=3}

\usepackage{geometry}
\usepackage{amsmath,amssymb}
\usepackage{%stmaryrd,
mathrsfs}
\usepackage{esint}

\allowdisplaybreaks[1]

\newtheorem{thm}{Theorem}[section]
\newtheorem{lm}[thm]{Lemma}

\newtheorem{prop}[thm]{Proposition}

\theoremstyle{definition}
\newtheorem{df}[thm]{Definition}
\newtheorem*{df*}{Definition}

\theoremstyle{remark}
\newtheorem{rem}[thm]{Remark}
\newtheorem*{rem*}{Remark}

\numberwithin{equation}{section}

\newcommand{\ci}[1]{_{ {}_{\scriptstyle #1}}}
\newcommand{\ti}[1]{_{\scriptstyle \text{\rm #1}}}
\newcommand{\Atd}{A_2^{\ensuremath\scriptstyle\textup{d}}}
\newcommand{\Mdnorm}[1]{\left[ #1 \right]_{A_2}^{{\scriptstyle\textup{d}}} }
\newcommand{\f}{\varphi}

\newcommand{\1}{\mathbf{1}}

\newcommand{\supp}{\operatorname{supp}}

\newcommand{\cB}{\mathcal{B}}

\newcommand{\cX}{\mathcal{X}}

\newcommand{\cz}{Calder\'{o}n--Zygmund\ }
\newcommand{\C}{\mathbb{C}}
\newcommand{\Z}{\mathbb{Z}}

\newcommand{\om}{\omega}

\newcommand{\E}{\mathbb{E}}

\newcommand{\be}{\mathbf{e}}
\newcommand{\bd}{\mathbf{d}}
\newcommand{\R}{\mathbb{R}}

\newcommand{\ch}{\operatorname{chld}}

\newcommand{\dom}{\operatorname{Dom}}
\newcommand{\spn}{\operatorname{span}}

\newcommand{\wt}{\widetilde}
\newcommand{\La}{\langle}
\newcommand{\Ra}{\rangle}
\newcommand{\N}{\mathbb{N}}
\newcommand{\cD}{\mathscr{D}}

\newcommand{\cL}{\mathcal{L}}

\newcommand{\bff}{\mathbf{f}}
\newcommand{\bg}{\mathbf{g}}
\newcommand{\bF}{\mathbf{F}}
\newcommand{\bG}{\mathbf{G}}
\newcommand{\bu}{\mathbf{u}}
\newcommand{\bv}{\mathbf{v}}

\newcommand{\sD}{{\scriptstyle\Delta}}
\newcommand{\BMOd}{\ensuremath\text{BMO}^{\textup d}}

\def\cyr{\fontencoding{OT2}\fontfamily{wncyr}\selectfont}
\DeclareTextFontCommand{\textcyr}{\cyr}
\newcommand{\sha}[0]{\ensuremath{\mathbb{S}%H\!\!S
}}

\newenvironment{entry}
{\begin{list}{X}%
  {%
      \setlength{\labelwidth}{55pt}%
      \setlength{\leftmargin}{\labelwidth}%\labelsep}%
      \addtolength{\leftmargin}{\labelsep}%
   }%
}%
{\end{list}}      

\renewcommand{\labelenumi}{(\roman{enumi})}

\newcounter{vremennyj}

\newcommand\cond[1]{\setcounter{vremennyj}{\theenumi}\setcounter{enumi}{#1}\labelenumi\setcounter{enumi}{\thevremennyj}}

\newcommand{\fdot}{\,\cdot\,}

\let\oldtocsection=\tocsection

\let\oldtocsubsection=\tocsubsection

\let\oldtocsubsubsection=\tocsubsubsection

\renewcommand{\tocsection}[2]{\hspace{0em}\oldtocsection{#1}{#2}}
\renewcommand{\tocsubsection}[2]{\hspace{1em}\oldtocsubsection{#1}{#2}}
\renewcommand{\tocsubsubsection}[2]{\hspace{2em}\oldtocsubsubsection{#1}{#2}}

\begin{document}

\title%[$A_2$ conjecture]
{Sharp $A_2$ estimates of Haar shifts via Bellman function}

\author{Sergei Treil}
\thanks{Supported  by the National Science Foundation under the grant  DMS-0800876. }
\address{Dept. of Mathematics, Brown University,   
151 Thayer
Str./Box 1917,      
 Providence, RI  02912, USA }
\email{treil@math.brown.edu}
\urladdr{http://www.math.brown.edu/\~{}treil}

\makeatletter
\@namedef{subjclassname@2010}{
  \textup{2010} Mathematics Subject Classification}
\makeatother

\subjclass[2010]{42B20, 42B35, 26B25, %47A30
60G42, 60G46}

% 42B	Harmonic analysis in several variables
% 42B20	Singular and oscillatory integrals (Calder?on-Zygmund, etc.)
% 42B35	Function spaces arising in harmonic analysis

% 47A	General theory of linear operators
% 47A30	Norms (inequalities, more than one norm, etc.)

%{30E20, 47B37, 47B40, 30D55.} 
%
% 30D55	$H^p$-classes (1980-2009)
% 30E20	Integration, integrals of Cauchy type, integral representations of analytic functions
%
% 47B   	Special classes of linear operators
% 47B37	Operators on special spaces (weighted shifts, operators on sequence spaces, etc.)
% 47B40	Spectral operators, decomposable operators, well-bounded operators, etc.

\keywords{\cz operators, $A_2$ weights, Haar shift, dyadic shift, Bellman function,
   non-homogeneous Harmonic Analysis, Harmonic Analysis on martingales}
\date{}

\begin{abstract}
We use  the  Bellman function method to give an elementary proof of a sharp weighted estimate for the Haar shifts, which is linear in the $A_2$ norm of the weight and  in the complexity of the shift.  Together with the representation of a general \cz operator as a weighted  average (over all dyadic lattices) of Haar shifts, cf.~\cite{H,HPTV-A2_2010}, it gives a significantly simpler proof of the so-called the $A_2$ conjecture. 

The main estimate (Lemma \ref{l-main_est}) is a very general fact about concave functions,  which can be very useful in other problems of martingale Harmonic Analysis. Concave functions of such type appear as the Bellman functions for bounds on the bilinear form of  martingale multipliers, thus the main estimate allows for the transference of the results for simplest possible martingale multipliers to more general martingale transforms.  

Note that (although this is not important for the $A_2$ conjecture for general \cz operators) this elementary proof gives the best known (linear) growth in the  complexity of the shift.
\end{abstract}

\maketitle

\setcounter{tocdepth}{1}
\tableofcontents

\section*{Notation}

\begin{entry}
\item[$\cD$] a dyadic lattice in $\R$ or $\R^d$;
\item[$\ch I$] the collection of children of the interval (cube) $I$;
\item[$\ch_k I$] the collection of children of the order $k$ of the interval (cube) $I$; the collection $\ch_0(I)$ consist of the interval $I$; 
\item[$|E|$] the $d$-dimensional Lebesgue measure of a set $E\subset \R^d$;
\item[$\La f \Ra \ci I$, $\fint_T f$] average of $f$ over $I$, $\La f \Ra \ci I = |I|^{-1} \int_I f(x) dx$;
\item[$E\ci I$] averaging operator, $E\ci I f := \La f\ci I \1\ci I$;
\item[$\Delta\ci I$] Martingale difference operator, $\Delta\ci I:= -E\ci I + \sum_{J\in\ch(I)} E\ci J$; 
\item[$L^2(w)$] the weighted $L^2$ space, $\| f\|\ci{L^2(w)}^2 = \int_{\R^n} |f(x)|^2 w(x) dx$. 
\end{entry}

\section{Introduction}
\label{intro}

The famous Hunt--Muckenhoupt--Wheeden theorem states that a \cz operator $T$ is bounded in the weighted space $L^2(w)=L^2(\R^d, w)$ if and only it the weight $w$ satisfies the so-called \emph{Muckenhoupt condition}
\begin{align*}
\label{A2}
\tag{$A_2$}
\sup_{Q} \left(|Q|^{-1} \int_Q w dx \right) \left(|Q|^{-1} \int_Q w^{-1} dx \right) =: [w]\ci{A_2} <\infty,
\end{align*}
where the supremum is taken over all cubes $Q$ in $\R^d$. The quantity $[w]\ci{A_2}$ is called the Muckenhoupt (or $A_2$) norm of the weight $w$ (although it is clearly not a norm). 

It has been an old problem to describe how the norm of a \cz operator in the weighted space $L^2(w)$ depends on the Muckenhoupt norm $[w]\ci{A_2}$ of $w$. A conjecture was that for a fixed \cz operator $T$ its norm in $L^2(w)$ is bounded by $C\cdot [w]\ci{A_2}$, where the constant $C$ depends on the operator $T$ (but not on the weight $w$). Simple counterexamples demonstrate that for the classical operators like Hilbert Transform or Riesz Transform, a better estimate than $C\cdot [w]\ci{A_2}$ is generally not possible. 

For operators that are not ``too singular'' better than linear estimates are possible. For example, for the convolution with the Poisson kernel $P_\tau$ its norm in $L^2(w)$ is estimated  by $C[w]\ci{A_2}^{1/2}$, i.e. 
\begin{align*}
\|P_\tau*f\|\ci{L^2(w)} \le C [w]\ci{A_2}^{1/2} \|f\|\ci{L^2(w)}. 
\end{align*}
This can be easily explained: the norm of the averaging operator $E\ci I$, $E\ci I f= \La f \Ra\ci I \1\ci I$ is exactly $\La w \Ra\ci I^{1/2} \La w^{-1} \Ra\ci I^{1/2}$, and the convolution with the Poisson kernel can be estimated by an average of the averaging operators. Also, it is not hard to show that this bound is sharp, meaning that for each weight $w$ convolution with some $P_\tau$ has the norm at least $c [w]\ci{A_2}^{1/2}$. 

But for the classical singular integral operators like Hilbert Transform or Riesz Transforms the linear in $[w]\ci{A_2}$ estimate is the best one can hope to get. The  conjecture about a linear in $[w]\ci{A_2}$  estimate of the norm for such and more general  operators has become known as the \emph{$A_2$ conjecture}. 

For the maximal function, the estimate  $C\cdot [w]\ci{A_2}$ was proved by S.~Buckley \cite{Buck1}: he also proved that this estimate is optimal for the maximal function. 
The first result for a ``singular integral'' operator was due to J.~Wittwer \cite{Wit}, who proved the  $A_2$ conjecture for the Haar multipliers. 

Using this result and Bellman function technique S.~Petermichcl and A.~Volberg \cite{PetmV} proved the $A_2$ conjecture for the Beurling--Ahlfors operator (convolution with $\pi^{-1} z^{-2}$ in $\C$). An alternative proof 
%of the $A_2$ conjecture 
was given a little later by O.~Dragi{\v{c}}evi{\'c} and A.~Volberg \cite{Dr-Volb-A2_Beurling_2003} via the representation of the Beurling--Ahlfors  Transform as an average of Haar multipliers over all dyadic lattices.

The result for other \cz operators, in particular for the Hilbert Transform (the convolution with $(\pi x)^{-1}$ on $\R$) remained open for some time. 

The next breakthrough was made by S.~Petermichl \cite{Petm1} who proved  the $A_2$ conjecture for the Hilbert Transform. She used the representation of the Hilbert Transform as the average over all (translated and dilated dyadic lattices) of a simple dyadic transformation (the so called Haar shift), and proved the $A_2$ conjecture for this operator. Later \cite{Petm2} she used similar ideas to prove the $A_2$ conjecture for the Riesz Transforms. 

Note that it took a lot of time and effort to go from J.~Wittwer's \cite{Wit} estimate for the simplest dyadic operator to the S.~Petermichl's  \cite{Petm1} estimate for the Haar shift, despite the fact that the Haar shift was just a little bit more complicated than the Haar multiplier. 

Lately, there was a lot of activity, that eventually lead to the complete solution of the $A_2$ conjecture for all \cz operators. Namely, M.~Lacey, S.~Petermichl and M.~Reguera \cite{LPR} proved the $A_2$ conjecture for all so-called Haar shifts, see Definitions \ref{df-Haar1}, \ref{df-Haar2} below, generalizations of  the operators used in \cite{Petm1}. This result immediately implies $A_2$ conjecture for all operators that can be represented as averages of Haar shifts of fixed \emph{complexity} (see Definition \ref{df-Haar1} below), in particular it gave another proof for Riesz transforms in $\R^d$. 

The technique used in \cite{LPR} was a very clever application of the stopping moment reasonings, and was based on the result from \cite{NTV5} about two-weight estimates for the so-called \emph{well-localized operators}. Unfortunately, the estimates of the norm in \cite{LPR} grew exponentially in the complexity of the Haar shift, so $A_2$ conjecture for general \cz operators remained open. 

For  general \cz operators the conjecture was finally settled by T.~Hyt\"{o}n\-en \cite{H}. 
One of the crucial components in his proof was the representation of an arbitrary \cz operator as a weighted  average (over all translated dyadic lattices) of the Haar shifts (of all possible complexities), where the weights decrease exponentially in the complexity%
\footnote{In fact, for general \cz operators one also should add the averages of so-called \emph{paraproducts}  
and their adjoint. But the paraproducts are the operators of fixed \emph{complexity}, 
so any known linear in $[w]\ci{A_2}$ 
estimate for such operators would work. We say more about estimates for paraproducts later in Section \ref{s-para}.}%
. 
This would imply the $A_2$ conjecture if one could prove the $A_2$ conjecture for the general Haar shifts with the estimates depending sub-exponentially (for example polynomially) on the complexity. 

This program was realized later in \cite{HPTV-A2_2010}, where the polynomial in complexity (and of course linear in $[w]\ci{A_2}$) estimate of the norm was obtained for general Haar shifts; a simpler representation of a \cz operator as a weighted average of the Haar shifts was also presented there. 

The original proof in \cite{H} used a result from a very technical paper \cite{PTV1}, where the $A_2$ conjecture was reduced to a Sawyer type testing condition and as a corollary to a weak type estimate. 

We should also mention as an interesting fact that all the proofs mentioned above used the results and/or technique from two papers by Nazarov--Treil--Volberg \cite{NTV-2w,NTV6}  about two weight estimates for martingale transforms

In the present paper we present a simple proof of the $A_2$ conjecture for arbitrary Haar shifts with \emph{linear} in complexity estimates. We get this estimate directly from the J.~Wittwer result \cite{Wit} ($A_2$ conjecture for the simplest Haar multipliers) using the Bellman function technique. 
Together with the representation of a \cz operator as a weighted average of Haar shifts from \cite{H} or \cite{HPTV-A2_2010} it gives the $A_2$ conjecture for general \cz operators. 

The proof is really simple and elementary; it is significantly simpler than any previous proof for the Haar shifts. 
It is really a shame that this proof was not discovered earlier. 

The essence of the proof is a simple result about convex functions. Functions of such type appear often when one uses Bellman function method in dyadic Harmonic analysis, so the main result can be used for the \emph{transference}, when one extends result obtained for the simplest dyadic model to more general martingale ones.  

Finally, we should also mention preprint \cite{RTV-2011} where the Bellman function method was used to prove the estimates for the Haar shifts of complexity $0$ and $1$ (in particular for the S.~Petermichl's Haar shifts \cite{Petm1}).

\section{Main objects.}

\subsection{Dyadic lattices}
\label{s:RDL}
The standard dyadic system in $\R^d$ is
\begin{equation*}
  \cD^{0}:=\bigcup_{k\in\Z}\cD^{0}_k,\qquad
  \cD^{0}_k:=\big\{ 2^{k}\big([0,1)^d+m\big):m\in\Z^d\big\}.
\end{equation*}
For $I\in\cD_k^0$ and a binary sequence $\om=(\om_j)_{j=-\infty}^{\infty}\in(\{0,1\}^d)^{\Z}$, let 
\begin{equation*}
  I\dot+\om:=I+\sum_{j<k}\om_j 2^{j}.
\end{equation*}
Following Nazarov, Treil and Volberg \cite[Section 9.1]{NTV5},  consider general dyadic systems of the form
\begin{equation*}
  \cD=\cD^{\om}:=\{I\dot+\om:I\in\cD^0\}
  =\bigcup_{k\in\Z}\cD^{\om}_k.
\end{equation*}
Given a cube $I=x+[0,\ell)^d$, let
\begin{equation*}
   \ch(I):=\{x+\eta\ell/2+[0,\ell/2)^d:\eta\in\{0,1\}^d\} 
\end{equation*}
denote the collection of dyadic children of $I$. Thus $\cD^{\om}_{k-1}=\bigcup\{\ch(I):I\in\cD^{\om}_k\}$. Note that, in line with \cite{NTV5,HPTV-A2_2010} but contrary to \cite{H}, we use the ``geometric'' indexing of cubes, where larger $k$ refers to larger cubes, rather than the ``probabilistic'' indexing, where larger $k$ would refer to finer sigma-algebras.

%Consider the standard probability measure on $\{0,1\}^d$, which 
%assigns equal probability $2^{-d}$ to every point. Define the 
%measure $\bP$ on $(\{0,1\}^d)^\Z$ as the corresponding product measure. 

\subsection{Martingale difference decompositions and Haar functions}

For a cube $I$ in $\R^d$ let 
\[
\E\ci I f := \left(\fint_I f dx\right) \1\ci I := \left(|I|^{-1}\int_I f dx\right) \1\ci I, \qquad \Delta\ci I :=-\E\ci I + \sum_{J\in \ch(I)} \E\ci{J} . 
\]
It is well known that for an arbitrary dyadic lattice $\cD$ every function $f\in L^2(\R^d)$ admits the orthogonal decomposition 
\[
f= \sum_{I\in\cD} \Delta\ci I f. 
\]

Given a cube $Q$ in $\R^d$, any function in the martingale difference space $\Delta\ci Q L^2$ is called a Haar function (corresponding to $Q$) and is usually denoted by  $h\ci Q$. Note, that here $h\ci Q$ denotes a \emph{generic} Haar function, not any particular one. 

In other words, a Haar function $h\ci Q$ is supported on $Q$, constant on the children of $Q$ and orthogonal to constants.

\subsection{Dyadic shifts}
\label{DS2w}

%\begin{df} An unweighted dyadic paraproduct is an operator $\Pi$ of the form
%\[
%\Pi f =\sum_{Q\in \cD} (\E\ci Q f )  h\ci Q\,,
%\]
%where $h\ci Q$ are some (non-weighted) Haar functions. 
%\end{df}

\begin{df}
\label{df-Haar1}
 Let $m, n\in \N$.  According to \cite{HPTV-A2_2010}   an elementary dyadic shift with parameters $m$, $n$ 
is an operator given by
\[
\sha f := \sum_{Q\in \cD}\sum_{\substack{ Q', Q''\in \cD,  Q', Q''\subset Q, \\ \ell(Q')=2^{-m}\ell(Q),\, \ell(Q'') =2^{-n}\ell(Q)}} |Q|^{-1}(f, h_{Q'}^{ Q''}) h_{Q''}^{Q'}
\]
where $h_{Q'}^{ Q''}$ and $h_{Q''}^{Q'}$ are (non-weighted) Haar functions for the cubes $Q'$ and $Q''$ respectively, 
subject to normalization
\begin{equation}
\label{norm1}
\|h_{Q'}^{ Q''}\|_\infty\cdot \|h_{Q''}^{Q'}\|_\infty \le 1.
\end{equation}
Notice that this implies, in particular, that
\begin{equation}
\label{sha1}
\sha f(x)=\sum_{Q\in \cD} |Q|^{-1} \int_Q a\ci Q(x,y)f(y) dy\,,\qquad \supp a\ci Q\subset Q\times Q, \ \|a\ci Q\|_{\infty}\le 1\,,
\end{equation}
where 
\begin{equation}
\label{sha-aQ}
a\ci Q (x,y) = \sum_{\substack{ Q', Q''\in \cD,  Q', Q''\subset Q, \\ \ell(Q')=2^{-m}\ell(Q),\, \ell(Q'')
=2^{-n}\ell(Q)}} h_{Q''}^{Q'}(x) h_{Q'}^{Q''}(y). 
\end{equation}
We will call the number $\max(m,n) +1$  the \emph{complexity} of the dyadic shift. Note that in \cite{H,HPTV-A2_2010} the complexity was defined as $\max(m,n)$; we use $\max(m,n) +1$, because it will be more convenient for our purposes. 
\end{df}

We will use a little more general definition of a Haar shift. 
\begin{df}
\label{df-Haar2}
A Haar shift $\sha$ of complexity $n$ is is given by 
\[
\sha f=\sum_{Q\in\cD} \sha\ci Q \Delta\ci Q^n f, 
\] 
where the operators $\sha\ci Q$  act on $\Delta^n\ci Q L^2$ and can be represented as integral operators with kernels $a\ci Q$, $\|a\ci Q\|_\infty\le |Q|^{-1}$. The latter means that for all $f, g\in \Delta\ci Q^n L^2$
\[
\La \sha \ci Q f, g\Ra = \int_Q a\ci Q(x, y) f (y) g(x) dx dy. 
\] 
\end{df}

We can always think that our dyadic shifts $\sha$ are {\it finite} dyadic shifts meaning that only finitely many $Q$'s are involved in its definition above. All estimates will be independent of this finite number.

\section{Bellman function for sharp weighted estimates of the dyadic martingale multipliers}

The simplest example of a Haar shift is the so-called dyadic martingale multiplier. Namely, let $\cD$ be the standard dyadic lattice in $\R$. For an interval $I\in\cD$ let $h\ci I$ be the standard $L^2$-normalized Haar function, $h\ci I := |I|^{-1/2} (\1\ci{I_1}-\1\ci{I_2})$, where $I_1$ and $I_2$ are the left and the right halves of $I$ respectively. Given a numerical sequence $\sigma = \{\sigma\ci I\}\ci{I\in \cD}$, $|\sigma\ci I|\le 1$,  define the operator $T_\sigma$ by 
\[
T_\sigma f = \sum_{I\in\cD} \sigma\ci I \La f, h\ci I\Ra h\ci I . 
\]
The family of operators $T_\sigma$ can be considered to be the simplest dyadic analog of singular integral operators, so the sharp weighted estimates for such operators were the natural thing to try before attacking the case of general \cz operators.

It was shown by J.~Wittwer \cite{Wit} that if a weight $w$ satisfies the dyadic Muckenhoupt condition
\[
\sup_{I\in\cD} \La w\Ra\ci I \La w^{-1}\Ra\ci I =: \Mdnorm w  <\infty 
\]
 the operators $T_\sigma$ are (uniformly in $\sigma$, $|\sigma\ci I|\le 1$) bounded in $L^2(w)$ by $C_1 \Mdnorm w$, where $C_1$ is an absolute constant.

Denoting by $I_1$ and $I_2$ he children of $I$ one can  rewrite the estimate of the norm as as 
\begin{align}
\label{eq-witt}
\sum_{I\in \cD} \left| \La f\Ra\ci{I_1} - \La f\Ra\ci{I_2} \right| \cdot \left| \La g\Ra\ci{I_1} - \La g\Ra\ci{I_2} \right|\cdot |I| \le C \Mdnorm w  \|f\|_{L^2(w)} \|g\|_{L^2(w^{-1})} ,  
\end{align}
for all $f\in L^2(w)$, $g\in L^2(w^{-1})$; here $C=C_1/4$. 

\subsection{Bellman function for the martingale multipliers and its properties}
\label{s-B_A}
Following the standard Bellman function technique, cf.~\cite{NT,NTV-name}, let us define the Bellman function for the problem. 
Let $A>1$, and let $I_1$ and $I_2$ denote the children of an interval $I$. Fix a dyadic interval $I_0$, for example $I_0=[0,1]$ and for real numbers $\bff, \bg, \bF, \bG, \bu, \bv$ satisfying
\begin{align}
\label{eq-domB}
	\bu, \bv > 0, \qquad 1 \le \bu\bv \le A, \qquad \bff^2 \le \bF \bv, \quad \bg^2 \le \bG \bu
\end{align}
define the function $\cB=\cB\ci A$ by
\[
\cB\ci A (\bff, \bg, \bF, \bG, \bu, \bv):= |I_0|^{-1}\sup \sum_{I\in \cD:I\subset I_0} \left| \La f\Ra\ci{I_1} - \La f\Ra\ci{I_2} \right| \cdot \left| \La g\Ra\ci{I_1} - \La g\Ra\ci{I_2} \right|\cdot |I|, 
\]
where the supremum is taken over all (real-valued) functions $f$, $g$ and the (dyadic) Muckenhoupt weights $w\ge 0$ on $I_0$,  such that 
\begin{align}
\label{eq-A_2^d-1}
 & \sup_{I\in\cD, I\subset I_0} \La w\Ra\ci I   \La w^{-1}\Ra\ci I   \le A ,
\\
\label{eq-av-fg}
	\La f\Ra\ci{I_0} &=\bff, \qquad \La f^2 w\Ra\ci{I_0} =\bF, \qquad
	\La g\Ra\ci{I_0} =\bg, \quad \La g^2 w^{-1}\Ra\ci{I_0} =\bG
	\\
\label{eq-av-uv}
	\La w \Ra\ci{I_0} & =\bu, \qquad \La w^{-1} \Ra\ci{I_0} = \bv. 
\end{align}
Here again $I_1$ $I_2$ are the children of the interval $I$. 

\subsubsection{Properties of \texorpdfstring{$\cB\ci A$}{B<sub>A}} 
\label{s-prop_b}
The Bellman function 
$\cB\ci A$ satisfies the following properties:
\begin{enumerate}
	\item Domain $\dom (\cB\ci A)$ of $\cB_A$ is given by \eqref{eq-domB}; this means that for every set of numbers $\bff$, $\bg$, $\bF$, $\bG$, $\bu$, $\bv$ satisfying \eqref{eq-domB} there are functions $f$, $g$ and a weight $w$ satisfying  \eqref{eq-av-fg}, \eqref{eq-av-uv}, so the supremum is well defined (not equal $-\infty$). It also mean, that if the variables are the corresponding averages, they must satisfy the constrains \eqref{eq-domB}. 
	
	\item Range: $0\le \cB\ci A (\bff, \bg, \bF, \bG, \bu, \bv) \le C A \bF^{1/2}\bG^{1/2}$;
	
	\item The main inequality: for any three $6$-tuples (let us call them $X=(\bff, \bg, \bF, \bG, \bu, \bv)$,  $X_1=(\bff_1, \bg_1, \bF_1, \bG_1, \bu_1, \bv_1)$, $X_2=(\bff_2, \bg_2, \bF_2, \bG_2, \bu_2, \bv_2)$) in the domain satisfying $X=(X_1 +X_2)/2$
	\[
	\cB\ci A (X) \ge  (\cB\ci A (X_1) + \cB\ci A(X_2))/2 + |\bff_1 - \bff_2|\cdot | \bg_1 - \bg_2|
	\]
\end{enumerate}
Note that because of correct homogeneity, the function $\cB\ti A$ does not depend on the choice of the interval $I_0$.

Let us explain the properties of the Bellman function $\cB\ci A$. The property \cond1 is easy to explain. Namely, for any weight $w$ and any interval $I$ 
\[
\La w \Ra\ci I \La w^{-1} \Ra\ci I \ge 1. 
\]
On the other hand the dyadic Muckenhoupt condition $\Atd$ means that $\La w \Ra\ci I \La w^{-1} \Ra\ci I \le A$ for any $I\in\cD$, so the inequalities $1\le\bu \bv \le A$ must be satisfied. On the other had, it is an easy exercise to show that for any $\bu\bv$ satisfying $1\le\bu\bv \le A$ one can find a dyadic  Muckenhoupt weight $w$ satisfying \eqref{eq-A_2^d-1} and \eqref{eq-av-uv}; one just can consider functions constant on the children of $I_0$. 

The inequality $\bff^2 \le \bF \bv$ is just the Cauchy--Schwartz inequality:
\begin{align*}
\left| \fint_I f dx \right| \le \left( \fint_I f^2 w dx\right)^{1/2} \left( \fint_I  w^{-1} dx\right)^{1/2}. 
\end{align*}
On the other hand, given a weight $w$ it is not hard to find a function $f$ satisfying \eqref{eq-av-fg}: we just put $f= \bu^{-1/2} \bff + \phi$, where $\int_{I_0} \phi w dx = 0$ and $\fint_{I_0} \phi^2 w = F - \bff^2 $. 

Property \cond2 is pretty straightforward: $\cB\ci A(X)\ge 0$ be the definition, and the inequality $\cB\ci A(X)\le C A\bF^{1/2}\bG^{1/2}$ follows from \eqref{eq-witt}. 

Let us now explain the main inequality \cond3. Let $X, X_1, X_2\in \dom(\cB\ci A)$, $X=(X_1+X_2)/2$, and let $I_1^0$ and $I_2^0$ be the children of the interval $I_0$. Consider functions $f$ and $g$ and a weight $w$ on $I_0$ such that  
\begin{align}
\label{eq-av-X}
X_{1,2}=(\La f \Ra\ci{I^0_{1,2}}, \La g \Ra\ci{I^0_{1,2}}, \La f^2 w \Ra\ci{I^0_{1,2}},  \La g^2 w^{-1}\Ra\ci{I^0_{1,2}},  \La w \Ra\ci{I^0_{1,2}},  \La w^{-1} \Ra\ci{I^0_{1,2}})
\end{align}
(we just construct the functions on the intervals $I^0_1$ and $I^0_2$ with the prescribed averages there: as we just discussed above such functions always exist). Then 
\[
X =(X_1+X_2)/2 = (\La f \Ra\ci{I_{0}}, \La g \Ra\ci{I_{0}}, \La f^2 w \Ra\ci{I_{0}},  \La g^2 w^{-1}\Ra\ci{I_{0}},  \La w \Ra\ci{I_{0}},  \La w^{-1} \Ra\ci{I_{0}})
\]
 is the vector of corresponding averages over $I_0$. Denoting by $I_1$ and $I_2$ the children of an interval $I$ we can write 
\begin{align*}
|I_0|^{-1} & 
\sum_{I\in\cD:\,I\subset I_0} \left| \La f\Ra\ci{I_1} - \La f\Ra\ci{I_2} \right| \cdot \left| \La g\Ra\ci{I_1} - \La g\Ra\ci{I_2} \right|\cdot |I|
\\
& =
|\bff_1-\bff_2|\cdot |\bg_1-\bg_2|
+ 
|I_0|^{-1} \sum_{I\in\cD:\,I\subsetneqq I_0} \left| \La f\Ra\ci{I_1} - \La f\Ra\ci{I_2} \right| \cdot \left| \La g\Ra\ci{I_1} - \La g\Ra\ci{I_2} \right|\cdot |I|
\end{align*}
Taking  the supremum over all $f$, $g$ and $w$ satisfying \eqref{eq-av-X} we get in the right side
\[
|\bff_1-\bff_2|\cdot |\bg_1-\bg_2|   + \left( \cB\ci A(X_1) + \cB\ci A(X_2)\right)/2. 
\]
The supremum in the left side is clearly bounded above by $\cB\ci A(X)$, which proves the main inequality \cond3. \hfill\qed
\begin{rem}
\label{r-non_conv}
Notice that the domain of $\cB\ci A$ is not convex, the reason being that the set $\{\bu, \bv>0: \bu\bv \le A\}$ is not convex. 

However, all the other constrains define the convex sets, so the non-convexity of the domain is completely given by the behavior of the coordinates $\bu$, $\bv$. Namely, if the points $X_+, X_-\in \R^6$ are in the domain, and for each $X=X_\theta= \theta X_+ + (1-\theta)X_-$, $0<\theta<1$, we know that $\bu\bv \le A$, then all $X_\theta$, $0<\theta <1$,  are in the domain. 
\end{rem}

\begin{rem}
The main inequality \cond3 implies that the functions $\cB\ci A$ are concave, namely that if the points $X_1$, $X_2$ and the whole interval $[X_1 , X_2] =\{(1-\theta) X_1 + \theta X_2: \theta \in [0,1]\}$ are in $\dom\cB\ci A$, then for all $\theta\in[0,1]$
\begin{align}
\label{eq-concavity}
	\cB\ci A ((1-\theta)X_1 + \theta X_2) \ge  (1-\theta) \cB\ci A (X_1) + \theta \cB\ci A(X_2). 
\end{align}

Condition \cond3 implies the so called \emph{midpoint concavity}, i.e.~\eqref{eq-concavity} with $\theta =1/2$.  But that is a well-known fact in convex analysis that for locally bounded concave (convex) functions midpoint concavity (convexity) is equivalent to the regular concavity (convexity), i.e.~to the inequality for all $\theta\in[0,1]$. 

The classical reference here would be the monograph \cite{Hardy-Inequalities-1952}, see Statement 111 in Section 3.18 where it was stated for convex functions. Note, that  only upper bound was assumed in \cite{Hardy-Inequalities-1952}, so for the equivalence of midpoint concavity and concavity one can only assume that a function is locally bounded below. 

Recall also, that any bounded above convex (respectively bounded below concave) function is continuous, and even locally Lipschitz, see \cite[Theorem 4.1.1]{Borwein_ConvexAn-2006}.

\end{rem}

\section{The main result: sharp weighted estimates of the Haar shifts}
 
The theorem below  is formally the main result of the paper. 

Let a dyadic lattice $\cD$ be fixed. Recall that a weight $w$ satisfies the dyadic Muckenhoupt condition $\Atd$ (with respect to the dyadic lattice $\cD$) if 
\begin{align*}
\sup_{I\in\cD} \La w\Ra\ci I \La w^{-1}\Ra\ci I =: \Mdnorm w  <\infty 
\end{align*}

\begin{thm}[Main result]
\label{t-main}
Let $\sha$ be a Haar shift in $\R^d$ 
%with respect to a dyadic lattice $\cD$ 
(in the sense of Definition \ref{df-Haar2}) of complexity $n$, and let a weight $w$ satisfies the dyadic Muckenhoupt condition $\Atd$. Then the norm of $\sha$ in $L^2(w)$ is at most $324 n 2^{2d-2} C \Mdnorm w$, where $C$ is the constant from \eqref{eq-witt}.  
\end{thm}

\section{Preliminaries for the proof}

\subsection{Some simple reductions}

Let us simplify the problem a bit.
 
\subsubsection{First slicing} 

Let us recall that the Haar shift $\sha$ was represented as 
\[
\sha f(x)=\sum_{Q\in \cD} |Q|^{-1} \int_Q a\ci Q(x,y)f(y) dy\,,\qquad \supp a\ci Q\subset Q\times Q, \ \|a\ci Q\|_{\infty}\le 1\,,
\]
where $a\ci Q$ are given by \eqref{sha-aQ}. 

Let $n$ be the complexity of the shift $\sha$. We can decompose $\sha$ as the sum $\sha = \sum_{k=0}^{n-1} \sha_k$, where  for $0\le k \le n$ the operator $\sha_k= \sha_k^n$ is defined by taking the sum only over the cubes $Q$ of length $2^{k+ nj}$, $j\in\Z$. If we denote $\cL_k=\cL_k^n:= \{Q\in\cD:\ell(Q) = 2^{k+nj}, j\in\Z\}$, then we can write 
\begin{align*}
\sha_k f = \sum_{Q\in \cL_k} |Q|^{-1} \int_Q a\ci Q(x,y)f(y) dy
\end{align*}
The advantage of the operators $\sha_k$ is that they are martingale transforms  when we are moving $n$ units of time at once, so it is possible to apply Bellman function method. 

For the readers who are not comfortable thinking in terms of martingales we will just write the estimate 
\begin{align}
\label{split1-est}
|\La \sha_k f, g \Ra | \le \sum_{Q\in\cL_k} |Q|^{-1} \|\Delta^n\ci Q f\|_1 \|\Delta^n\ci Q g \|_1
\end{align}
and notice that the functions $\Delta^n\ci Q f$ and $\Delta^n\ci Q g$ are constant on all cubes $R\in \cL_k$, $R\subsetneqq Q$.

\subsubsection{Reduction to the real line}
\label{s-red_to_R}
 First of all it is sufficient to deal  only  with the Haar shifts on a dyadic lattice in $\R$ (we can assume that we are dealing with the standard dyadic lattice, but the proof for a general dyadic lattice is the same). 

\setlength{\unitlength}{1mm}
\begin{figure}

\begin{center}
\includegraphics{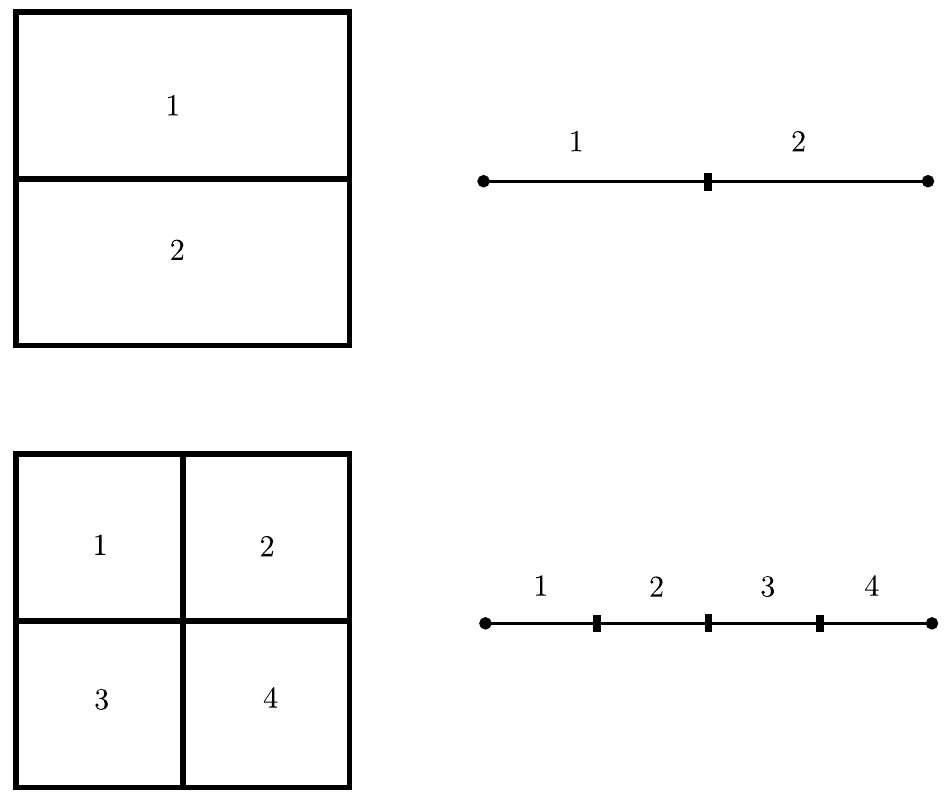}
\end{center}
\caption{Putting cubes on the line}%
{\protect\label{fig1}}%
\end{figure}

So for the slice $\sha_k$ of the Haar shift in $\R^d$ we make its  representation on the real line by ``arranging'' cubes along the real line.
% (i.e.~by using instead of a cube $Q$ an interval $I$, $|Q|=|I|$). 
Namely, for a dyadic cube $Q$ in $R^d$ take a dyadic  interval $I$, $|I|=|Q|$; this interval will correspond to the cube $Q$.  Dividing one side of $Q$ into $2$, split $Q$ into $2$ equal parallelepipeds: then pick a (one to one) correspondence between the parallelepipeds and the children of $I$ (the choice of the correspondence does not matter). 

Then by dividing  a longer side split each parallelepiped into two equal ones, and  and make a correspondence between the new parallelepipeds and the children of the corresponding intervals, see Fig.~\ref{fig1}. 

After $d$ divisions  we end up with the correspondence between the children of $Q$ and the intervals $J\in \ch_{d}(I)$. Note that the intervals $J\in\ch_k(I)$, $1\le k< d$ correspond to some ``almost children'' $R$ of $Q$. Here by an ``almost child'' we mean a parallelepiped some with some of the sides coinciding with the sides of $Q$ and the other sides being halves of the corresponding sides of $Q$. The construction for $d=1$ is presented on  Fig.~\ref{fig1}. 

One can run this construction up, i.e.~for $\wt I$ being the parent of $I$ and for $\wt Q$ being the grandparent of $Q$ of order $d$ we can construct, using the procedure described above,  a bijection $\Phi$ from the children and almost children of $\wt Q$ to the intervals $J\in \ch_k(\wt I)$, $1\le k<d$, such that $\Phi(Q)=I$. To assure  that $\Phi(Q)=I$ one just need at every division make the image of the almost child containing $Q$ to be the dyadic interval (of the appropriate length) containing $I$. 

A  function  $f\in L^1\ti{loc}(\R^d)$ will be transferred to a function $g\in L^{1}\ti{loc}(\R)$ such that $\La f\Ra\ci Q = \La g\Ra\ci I$, for all $Q, I$, $I=\Phi(Q)$. 

Let us see what is the ``price'' to pay for this reduction. First, if $\sha$ is a Haar shift (or its slice) of a complexity $n$ in $\R^d$, then  its model in $\R$ will be a shift of complexity $nd$. 

The (dyadic) $A_2$ norm of the weight on $\R$ is  $\sup_R \La w\Ra\ci R  \La w\Ra\ci R$, where the supremum is taken over all dyadic cubes and over all their almost children.  If $R$ is an almost child of $Q$, then
\[
\int_R w(x) dx \le \int_Q w(x) dx, \qquad \int_R w(x)^{-1} dx \le \int_Q w(x)^{-1} dx, 
\]
and $|R|\ge 2^{-d+1}|Q|$. So after the transfer to the real line, the Muckenhoupt norm $[w]\ci{A_2^\cD}$ of the weight increases at most $2^{2(d-1)}$ times. 

% is to replace the estimate of the $A_2$ norm of the weight from  
% $A$ to $2^{d-1}A$. Also, to the shift of complexity 
% $n$ in $\R^d$ the complexity of the model on $\R$ will be $dn$. 

\subsection{A technical lemma}

\begin{lm}
\label{l-dbl_domain}
Let $X, X_+, X_-\in \dom \cB\ci A$, $X= (X_+ +X_-)/2$. Then the interval $[X_-,X_+]$ belongs to $\dom \cB_{A'}$, where $A'=9A/8$, i.e.~for all $\theta \in (0,1)$ we have $X_\theta =\theta X_+ + (1-\theta) X_- \in\dom\cB\ci{A'}$. 
\end{lm}
\begin{proof}
As it was discussed above in Remark \ref{r-non_conv}, we only need to check the non-convex constrain $\bu\bv \le A$. Because of scaling, it is sufficient to show that if end points and a center of an interval are in the set $\{(x,y)\in\R^2: x, y\ge 0, xy\le 1\}$, then for any point on the interval $xy\le9/8$. 

It is no hard to see that the worst case scenario is when center and one endpoint is on the line $y=1/x$ (and the other endpoint is on a coordinate axis). Since the rescaling $x\mapsto \alpha x$, $y\mapsto \alpha^{-1} y$, $\alpha>0$ does not change the domains $xy\le C$, we can assume without loss of generality that the center of the interval is at the point $(1,1)$. Then the picture for worst case scenario will be as on Fig.~\ref{fig2}, or its reflection in the line $y=x$.

\begin{figure}
\begin{center}
\includegraphics{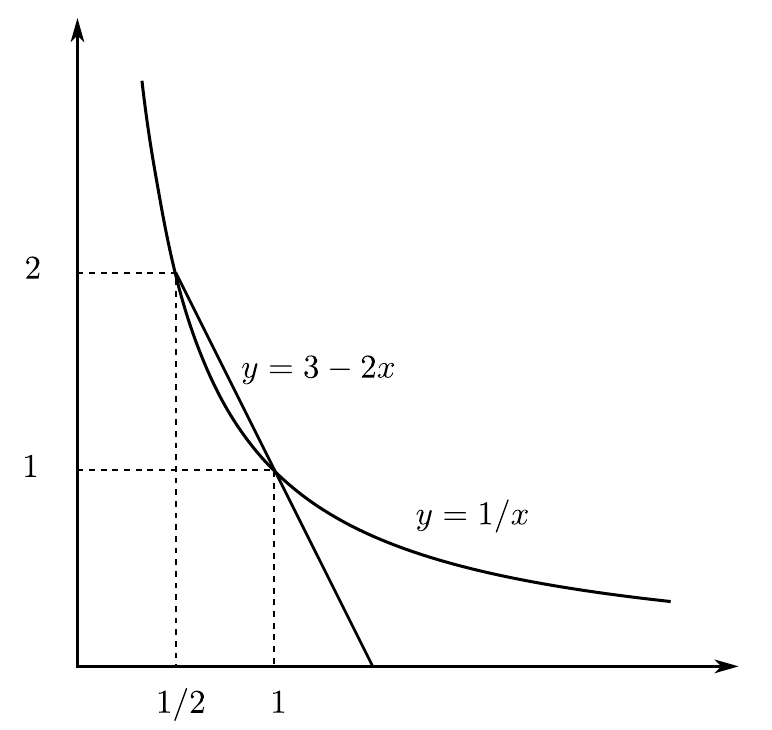}
\end{center}
\caption{Proof of Lemma \ref{l-dbl_domain}}%
{\protect\label{fig2}}%
\end{figure}

It is an easy exercise  to find that the maximal value of $xy$ on this interval is $9/8$. 
\end{proof}

\section{The main estimate}
The following lemma is in the heart of the matter. In fact, it can be considered as the main result of the paper.
\begin{lm}[The main estimate]
\label{l-main_est}
Let $\cB_A$, $A>1$ be a family of functions satisfying conditions \cond1, \cond2, \cond3 from Section \ref{s-B_A}.  
Let $I_0$ be a dyadic interval, and let for all $I\in\ch_k(I_0) $, $0\le k\le n$, the points $X\ci I = (\bff\ci I, \bg\ci I, \bF\ci I, \bG\ci I, \bu\ci I, \bv\ci I) \in \dom \cB\ci A$ be given. 
Assume that $X\ci I$ satisfy the dyadic martingale dynamics, namely that if $I_1$, $I_2$ are the children of $I$, then 
\begin{align*}
X\ci I = \left(X\ci{I_1} + X\ci{I_2} \right)/2. 
\end{align*}
Then for $A'=4.5A$ we have 
\begin{align*}
\left(2^{-n} \sum_{I\in \ch_n(I_0)} | \bff\ci I -\bff\ci{I_0} | \right)
&\left(2^{-n} \sum_{I\in \ch_n(I_0)} | \bg\ci I -\bg\ci{I_0} | \right)
\\ &
\le
72 \left( \cB\ci{A'} (X\ci{I_0}) - 2^{-n} \sum_{I\in \ch_n(I_0)} \cB\ci{A'} (X\ci I) \right). 
\end{align*}
%%
%where $C$ is an absolute constant. 
\end{lm}

\subsection{Plan of the proof} Let us first explain the idea of the proof informally. We will use the language of random processes (martingales) as the most convenient for the informal explanation, but the formal proof will be completely elementary. 

So, we a given a dyadic martingale $X\ci I$ $I\in\cD$, where at each point we have 2 choices with equal probability,  and we want to estimate below the difference between $\cB\ci{A'}(X_I)$ at the initial moment $I=I_0$ and its expected value after $n$ steps. 

An immediate idea would be to use the property \cond3 of the Bellman function to estimate expected loss at each step. But this would not work: the conditional expectation of the loss when going from the state $X\ci I$ to the states $X\ci{I_1}$, $X\ci{I_2}$, where $I_{1,2}$ are the children of $I$  is estimated below by $c|\bff\ci{I_1} -\bff\ci{I}| \cdot |\bg\ci{I_1} -\bg\ci{I}|$, and this quantity can be small, it even can be zero: imagine a situation when at each step only one of the variables $\bff$ or $\bg$ is changed. 

So, the main idea is to change the martingale, preserving the starting point $X\ci{I_0}$ and the final values and  probabilities. 

Namely, from the starting point $X\ci{I_0}$ we move with probabilities $1/2$ to the points $X^+$, and $X^-$, $X^{\pm}=X^\pm\ci{I_0}=(\bff^{\pm}, \bg^{\pm}, \bF^{\pm}, \bG^{\pm}, \bu^{\pm}, \bv^{\pm})$, $X\ci{I_0} =(X^+ + X^-)/2$. We will move ``sufficiently far'', so  $|\bff^\pm -\bff\ci{I_0}|$ and $|\bg^\pm - \bg\ci{I_0}|$ are comparable to  $2^{-n} \sum_{I\in \ch_n(I_0)} | \bff\ci I -\bff\ci{I_0} |$ and $2^{-n} \sum_{I\in \ch_n(I_0)} | \bg\ci I -\bg\ci{I_0} |$ respectively, but ``not too far'', so the points $X^\pm$ will be ``almost averages'' of the points $X\ci I$, $I\in\ch_n(I_0)$.  
% after $n$ steps. 

Since we moved sufficiently far,  we get the amount we need at the first step. 
Moreover, since    the points $X^\pm$ are ``almost averages'' of the points $X\ci{I}$, $I\in \ch_n(I_0)$, we can start the process from the points $X^\pm$ to end up after $n$ steps at the points $X\ci I$, $I\in\ch_n(I_0)$. 

Of course, it will not be possible to get from $X^+$ (or $X^-$) to the end points $X\ci I$, $I\in\ch_n(I_0)$ via the standard dyadic martingale, when at each point we have 2 choices with equal probability. But since the first step was sufficiently small, it will still be possible to get to the desired endpoints via a binary martingale, where at each point we have 2 choices with not equal but with \emph{almost} equal probability.  

It might be intuitively clear from the symmetry of $X^\pm$, that after moving from $X\ci{I_0}$ to $X^\pm$ it is possible to start the martingale from $X^\pm$ so we get from $X\ci{I_0}$ to the endpoints with equal probability.  More precisely, we will get from $X^+$ to $X\ci I$, $I\in\ch_n(I_0)$ with probability $2^{-n} (1+\alpha\ci I)$, $|\alpha\ci I|\le 1/3$, and we get from $I^-$ to the same point $I$ with the probability $2^{-n}(1-\alpha\ci I)$, so the probability of getting from $I_0$ to $I$ will be  $2^{-n}$.  

If this reasoning is not intuitively clear to the reader, he or she should not worry, because the rigorous proof (not requiring any probability) will be presented later; the probabilistic interpretation will guide us through the formal calculations.  

\subsection{Formal proof: another technical lemma}

Let us recall a theorem by K.~Ball \cite[Theorem 7]{Ball_PlankProblem-1991}, solving the famous Tarski \emph{plank problem} for convex bodies. This theorem also can be treated as a ``multiple Hahn--Banach Theorem''. 

\begin{thm}
\label{t-plank}
Let $x_k$ be unit vectors ($\|x_k\|=1$) in a (real) normed space $\cX$, and let $m_k\in \R$, $w_k\ge 0$ such that 
\[
\sum_{k} w_k =1. 
\]
Then there exist a functional $x^*\in\cX^*$, $\|x^*\|\le 1$ such that 
\[
\left| \La x_k, x^*\Ra - m_k \right| \ge w_k \qquad \text{for all } k. 
\]
\end{thm}
Note that the number of vectors here can be infinite. 

This theorem with $2$ vectors, $m_k=0$, $w_1=w_2=1/2$ gives as the following simple lemma that we will use.

\begin{lm}
\label{l-shift1}
Let $\cX$ be a (real) normed space, and let $a, b\in\cX$, $\|a\|=\|b\|=1$. There exists $x^*\in\cX^*$, $\|x^*\|=1$ such that
\[
| \La a, x^*\Ra| \ge 1/2, \qquad | \La b, x^*\Ra| \ge 1/2. 
\]
\end{lm}

Note, that for our purposes we just need this lemma with some constant, not necessarily with the optimal one $1/2$; we just get a worse constant in Theorem \ref{t-main}. A proof of this lemma with some constant is an easy exercise in elementary functional analysis.

\subsection{Formal proof: the first step}
Let $N=2^n$, and let $\ell^p_N$ be the space $\R^N$ %(or $\C^N$) 
endowed with the $\ell^p$ norm. Define $\be\in \ell^p_N$, $\be = (1, 1, \ldots, 1)$. 
Consider the quotient space $\cX = \ell^1_N/\spn\{\be\}$. For $x\in \ell^1_N$ let $x^0:= x- N^{-1} \La x, \be \Ra \be$. Then
\begin{align}
\label{eq-QuotNorm}
\|x\|\ci{\cX} \le \|x^0\|_{\ell^1_N} \le 2\|x\|\ci{\cX}.	
\end{align}
Indeed, the first inequality is trivial (follows from the definition of the norm in the quotient space). As for the second one, $|\La x, \be\Ra| \le  \|x\|_{\ell^1_N}$, $\|\be\|_{\ell^1_N} =N$, so it follows from the triangle inequality that
\[
\|x^0\|_{\ell^1_N} \le \|x\|_{\ell^1_N} + N^{-1} |\La, x, \be \Ra | \|e\|_{\ell^1_N} \le 2 \|x\|_{\ell^1_N}. 
\] 
This inequality remains true if one replaces $x$ by $x-\alpha\be$, $\alpha\in\R$, so the second inequality in \eqref{eq-QuotNorm} is proved. 

The dual space $\cX^*$ can be identified with s subspace of $\ell^\infty_N$ consisting of $x^*\in \ell^\infty_N$ such that $\La \be, x^*\Ra =0$ (with the usual $\ell^\infty_N$-norm).

%It follows from \eqref{eq-QuotNorm} that 

Applying Lemma \ref{l-shift1} to the above space $\cX$ with $a$ and $b$ being the normalized vectors $\{\bff\ci I-\bff\ci{I_0}\}\ci{I\in\ch_n(I_0)}$ and $\{\bg\ci I-\bg\ci{I_0}\}\ci{I\in\ch_n(I_0)}$ respectively, we get that there exists $x^*=\{\alpha\ci I\}_{I\in \ch(I_0)}$ such that 
\begin{align}
\label{sum0}
&\sum_{I\in\ch_n(I_0)} \alpha\ci I   =0,  \\
\label{bd_alpha_I}
 &|\alpha\ci I| \le 1/3 \qquad \forall I\in\ch_n(I_0), 
\end{align}
and 
\begin{align}
\label{move-f}
\left| \sum_{I\in\ch_n(I_0)} \alpha\ci I ( \bff\ci I-\bff\ci{I_0}) \right| & \ge \frac1{12} \sum_{I\in\ch_n(I_0)} |\bff\ci I-\bff\ci{I_0}|, 
\\
\label{move-g}
\left| \sum_{I\in\ch_n(I_0)} \alpha\ci I ( \bg\ci I-\bg\ci{I_0}) \right| & \ge \frac1{12} \sum_{I\in\ch_n(I_0)} |\bg\ci I-\bg\ci{I_0}| . 
\end{align}
Note that Lemma \ref{l-shift1} gives the comparison with the norm in $\cX=\ell^1_N/\spn{\be}$ with the constant $1/6$, and the second inequality in \eqref{eq-QuotNorm} give as an extra factor $1/2$.

Below, all identities (inequalities) with $\pm$ mean that there are two sets of identities (inequalities): one with $+$ and the other with $-$. 

Define  $X^{\pm}=X^\pm\ci{I_0}=(\bff^{\pm}, \bg^{\pm}, \bF^{\pm}, \bG^{\pm}, \bu^{\pm}, \bv^{\pm})$ by
\begin{align*}
X^\pm =2^{-n} \sum_{I\in\ch_n(I_0)} (1 \pm \alpha\ci I) X\ci I
=X\ci{I_0} \pm 2^{-n} \sum_{I\in\ch_n(I_0)}  \alpha\ci I X\ci I
,
%\qquad 
%X^2 =2^{-n} \sum_{I\in\ch_n(I_0)} (1+\alpha\ci I) X\ci I,  
\end{align*}
so $X\ci{I_0} = (X^+ +X^-)/2$. 

Note that \eqref{move-f}, \eqref{move-g} and the identities 
\[
\bff\ci{I_0} = 2^{-n}  \sum_{I\in\ch_n(I_0)} \bff\ci I, \qquad 
\bg\ci{I_0} = 2^{-n}  \sum_{I\in\ch_n(I_0)} \bg\ci I
\]
imply that 
\begin{align*}
|\bff^\pm - \bff\ci{I_0}| 
%&= |\bff^2-\bff\ci{I_0}| 
&\ge \frac{1}{12} 2^{-n} \sum_{I\in\ch_n(I_0)} |\bff\ci I-\bff\ci{I_0}|, 
\\
|\bg^\pm - \bg\ci{I_0}| 
%&= |\bg^2-\bg\ci{I_0}| 
&\ge \frac{1}{12} 2^{-n} \sum_{I\in\ch_n(I_0)} |\bg\ci I-\bg\ci{I_0}|.
\end{align*}
Let
\[
x\ci I^\pm := 1\pm\alpha\ci I;
\]
note that $2/3\le x\ci I \le 4/3$. 

\subsection{Formal proof: modifying the martingale} For $I\in\ch_k(I_0)$, $1\le k <n$ define
\begin{align}
\label{X_I^pm}
X\ci{I}^\pm := \left(\sum_{\substack{J\in \ch_n(I_0): \\J\subset I}} x\ci J^\pm   X\ci J \right)
\div
\left(\sum_{\substack{J\in \ch_n(I_0): \\J\subset I}} x\ci J^\pm   \right). 
\end{align}
Note, that for $I\in\ch_n(I_0)$ we have $X^+\ci I = X^-\ci I = X\ci I$, where $X\ci I$ are the points given in the statement of the lemma.

For $I\in\ch(K)$ define 
\begin{align}
\label{theta_I^pm}
\theta\ci{I}^\pm := \left(\sum_{\substack{J\in \ch_n(I_0): \\J\subset I}} x\ci J^\pm    \right)
\div
\left(\sum_{\substack{J\in \ch_n(I_0): \\J\subset K}} x\ci J^\pm   \right)
\end{align}
Note that $\theta\ci I^\pm \ge 0$, and it $I_1$, $I_2$ are the children of $I$, then 
\begin{align}
\label{midpoint}
\theta\ci{I_1}^\pm +\theta\ci{I_2}^\pm  =1, \qquad \theta\ci{I_1}^\pm X^\pm\ci{I_1}+ \theta\ci{I_2}^\pm X^\pm\ci{I_2} =  X^\pm\ci I, 
\end{align}
i.e.~$X^+_I$ (respectively $X\ci{I}^-$) is in the interval connecting $X\ci{I_1}^+$ and $X_{I_2}^+$ (respectively $X\ci{I_1}^-$ and $X\ci{I_2}^-$). From the probabilistic point of view, $\theta^\pm\ci{I_1}$ and $\theta^\pm\ci{I_2}$ are the probabilities of moving from $X^\pm\ci I$ to $X^\pm\ci{I_1}$ and $X\ci{I_2}$ respectively. 

It is not hard to show that $\theta^\pm\ci{I_{1,2}}$ cannot be too close to $0$ or $1$, but we do not need this fact for the formal proof.

Identity \eqref{X_I^pm} means that the points $X_I^\pm$ are in the convex hull of the points $X\ci J : J\in \ch_n(I_0), J\subset I$. We claim that moreover, $X\ci I^\pm \in \dom(\cB\ci{4A})$ for all $I\in\ch_k(I_0)$, $0\le k \le n$.

Indeed, since $X^\pm\ci I$ are in the convex hull of the points $X\ci J \in \dom(\cB\ci A) \subset \dom \cB\ci{4A}$,  
%$J\in\ch_n(I_0)$, 
and among the constrains defining $\dom(\cB\ci{4 A})$ only the constrain 
\[
\bu \bv \le 4A
\]
is not convex, we only need to check this constrain.  

Notice that equation \eqref{X_I^pm} with all $x^\pm\ci I$ replaced by $1$ 
%We know that if $x^\pm\ci J =1$ for all $J$, then \eqref{X_I^pm} 
gives us $X\ci I$, which belongs to $\dom(\cB\ci A)$. Let us look at the $\bu$-coordinate of $X\ci I^\pm$: we can easily see that 
\[
\bu\ci I^\pm \le \bu\ci I \frac43/\frac23= 2\bu\ci I, 
\]
because the maximal possible numerator in the $\bu$-coordinate of \eqref{X_I^pm} happens when all $x\ci J^\pm = 4/3$  and the minimal possible numerator when all $x\ci J^\pm =2/3$. The same holds for $\bv\ci I^\pm$, so $\bu\ci I^\pm \bv\ci I^\pm \le 4\bu\ci I \bv\ci I \le 4A$, thus $X\ci I^\pm\in\dom(\cB\ci{4A})$. 

We can also show that if $I_1$ and $I_2$ are the children of $I$, then the centers of the intervals $[X^\pm\ci{I_1}, X^\pm\ci{I_2}]$ also belong to $\cB\ci{4A}$. Indeed, we know that  $X\ci I =(X\ci{I_1}+X\ci{I_2})/2$ (the center of $[X\ci{I_1}, X\ci{I_2}]$) belongs to $\cB\ci A$. As we discussed above, 
\[
\bu^\pm\ci{I_{1,2}} \le 2 \bu\ci{I_{1,2}}, \qquad 
\bv^\pm\ci{I_{1,2}} \le 2 \bv\ci{I_{1,2}} .
\]
Therefore,  if $\wt\bu^\pm$ and $\wt\bv^\pm$ are the $\bu$ and $\bv$ coordinates of the centers of the intervals $[X^\pm\ci{I_1}, X^\pm\ci{I_2}]$, we can conclude that 
\[
\wt\bu^\pm \le 2\bu\ci I, \qquad \wt\bv^\pm \le 2\bv\ci I
\]
thus the centers of  the intervals $[X^\pm\ci{I_1}, X^\pm\ci{I_2}]$ also belong to $\cB\ci{4A}$ (as we discussed above, we only need to check the non-convex constrain $\bu\bv \le 4A$). 

Therefore, by Lemma \ref{l-dbl_domain} the intervals $[X^\pm\ci{I_1}, X^\pm\ci{I_2}]$ are in $\cB\ci{A'}$, where $A'=4A (9/8)=4.5A$.

\subsection{Conclusion of the proof} Now we are ready to complete the proof. By property \cond3 of the Bellman function, using the fact that $|\bff^+ -\bff^- | = 2|\bff^\pm-\bff\ci{I_0} |$
\begin{align}
\label{first_step}
|\bff^\pm - \bff\ci{I_0}|\cdot |\bg^\pm - \bg\ci{I_0}| \le \frac14 \left( \cB\ci{A'}(X\ci{I_0}) - \frac12 \left( \cB\ci{A'}(X^+ )  + \cB\ci{A'}(X^- ) \right) \right)
\end{align}
It follows from the concavity of $\cB\ci{A'}$ and \eqref{midpoint} that if $I_1$ and $I_2$ are the children of $I$, then 
\begin{align*}
\cB\ci{A'} (X\ci I^\pm) \ge \theta^\pm\ci{I_1}\cB\ci{A'} (X\ci{I_2}^\pm)
+ \theta^\pm\ci{I_2}\cB\ci{A'} (X\ci{I_2}^\pm);
\end{align*}
recall that as we discussed above, 
%%all the points $X\ci I^\pm$ are in $\dom(\cB\ci{4A})$, so by Lemma \ref{l-dbl_domain} 
the whole intervals $[X\ci{I_1}^\pm, X\ci{I_2}^\pm]$ are in $\dom(\cB\ci{8A})$. 

Let us apply this inequality to $I_0$, then substitute in the right side the inequalities for the children of $I_0$, and so on. Then,  using the fact that for $I\in\ch_n(I_0)$
\begin{align*}
 \prod_{J\in \cD: I\subset J \subsetneqq I_0} \theta^\pm\ci J =x^\pm\ci I \div 
\left( \sum_{J\in\ch_n(I_0)} x^\pm\ci J\right) = x^\pm\ci I 2^{-n}, 
\end{align*}
we get the inequality
\begin{align*}
\cB\ci{A'}(X^\pm) \ge 2^{-n} \sum_{I\in\ch_n(I_0)} x^\pm\ci I \cB\ci{A'}(X\ci I);
\end{align*}
note that we are using $X\ci I$ instead of $X^\pm\ci I$ in the right side, because $X^+\ci I = X^-\ci I =X\ci I$ for $I\in\ch_n(I_0)$.  

Substituting this inequality into \eqref{first_step} and taking into account that $x^+\ci I + x^-\ci I =2$ for $I\in \ch_n(I_0)$, we get 
\begin{align}
\label{final_diff_1}
|\bff^\pm - \bff\ci{I_0}|\cdot |\bg^\pm - \bg\ci{I_0}| \le
\frac14
\left( \cB\ci{A'} (X\ci{I_0}) - 2^{-n} \sum_{I\in \ch_n(I_0)} \cB\ci{A'} (X\ci I) \right)
\end{align}

The estimates \eqref{move-f} and \eqref{move-g} mean that 
\begin{align}
\label{move-f1}
	2^{-n}  \sum_{I\in\ch_n(I_0)} |\bff\ci I-\bff\ci{I_0}| & \le 12 \,|\bff^\pm - \bff\ci{I_0}|,
	\\
\label{move-g1}
	2^{-n}  \sum_{I\in\ch_n(I_0)} |\bg\ci I-\bg\ci{I_0}|  &\le 12 \,|\bg^\pm - \bg\ci{I_0}|.
\end{align}
Combining these estimate with \eqref{final_diff_1} we get the conclusion of the lemma. 
\  \hfill\qed

\section{Proof of the main result \texorpdfstring{(Theorem \ref{t-main}) }{} }
As we discussed above, we need to estimate the ``slices'' $\sha_k$ of the Haar shift, or equivalently, their models on the real line $\R$. Let $\sha=\sha_k$ be such a model of a slice, and let $n$ be its complexity.  As it was shown before
\begin{align*}
|\La \sha_k f, g \Ra | \le \sum_{I\in\cL_k} |I|^{-1} \|\Delta^n\ci I f\|_1 \|\Delta^n\ci I g \|_1
\end{align*}

For $I\in\cD$ let $X\ci I:=(\La f \Ra\ci I, \La g \Ra\ci I, \La f^w \Ra\ci I, \La g w^{-1} \Ra\ci I, \La w \Ra\ci I, \La w^{-1} \Ra\ci I )$.  

Define $A:= \Mdnorm w$, so $X\ci I\in \dom(\cB\ci A)$ for all $I\in\cD$. 
Lemma \ref{l-main_est} states that 
\begin{align*}
|I|^{-1} \|\Delta^n\ci I f\|_1 \|\Delta^n\ci I g \|_1  \le |I|\cB\ci{A'} (X\ci I) - \sum_{J\in\ch_n(I)} |J|\cB\ci{A'} (X\ci J)
\end{align*}
Writing this estimate for each $J\in\ch_n(I)$, then repeating this $m$ times we get 
\begin{align*}
\sum_{\substack{J\in\cL_k : J\subset I\\    \ell(J) >  2^{-nm} \ell(I) } } 
|J|^{-1} \|\Delta^n\ci J f\|_1 \|\Delta^n\ci J g \|_1  
& \le  
72\left( |I|\cB\ci{A'} (X\ci I) - \sum_{J\in\ch_{mn}(I)} |J|\cB\ci{A'} (X\ci J) \right)
\\
& \le 
4.5 \cdot 72\cdot CA\La f^2 w\Ra\ci I^{1/2}  \La g^2 w^{-1}\Ra\ci I^{1/2} |I|
\\
& = 324 CA \|f\1\ci I\|_{L^2(w)} \|g\1\ci I\|_{L^2(w^{-1})}
\end{align*}
where $C$ is the constant from the property \cond2 of the Bellman functions $\cB\ci A$; here the second inequality holds because $\cB_{A'}(X\ci I)\le  CA' \bF^{1/2}\ci I\bG^{1/2}\ci I =4.5 CA \bF^{1/2}\ci I\bG^{1/2}\ci I$ and $\cB\ci{A'}(X\ci J)\ge 0$ by property \cond2 of the Bellman function.  

Letting $m\to \infty$ we get 
\begin{align*}
\sum_{J\in\cL_k : J\subset I }  |J|^{-1} \|\Delta^n\ci J f\|_1 \|\Delta^n\ci J g \|_1
\le 
324CA \|f\1\ci I\|_{L^2(w)} \|g\1\ci I\|_{L^2(w^{-1})}
\end{align*}
Covering the line by the intervals $I\in\cL$ of length $2^M$ and applying the above inequality to each $I$ we get 
\begin{align*}
\sum_{J\in\cL_k : |J|\le 2^M }  |J|^{-1} \|\Delta^n\ci J f\|_1 \|\Delta^n\ci J g \|_1
\le 
324CA \|f \|_{L^2(w)} \|g\|_{L^2(w^{-1})}, 
\end{align*}
and letting $M\to\infty$ we get that the norm of each slice $\sha_k$ is bounded by $324 C A$. Recall, that we had $n$ slices. Recall that $A$ is the $A_2$ norm of the weight $w$ transferred to the real line $\R$. As it was discussed above in Section \ref{s-red_to_R}, it is estimated by $2^{2d-2} [w]\ci{A_2}$, where  $[w]\ci{A_2}$ is the $A_2$ norm of the original weight $w$ in $\R^d$. 

Gathering everything together we get the conclusion of the main result (Theorem \ref{t-main}).

\section{Estimates of the paraproducts}
\label{s-para}

As it was mentioned before in the Introduction, to prove the $A_2$ conjecture for general \cz operators, besides getting linear in $\Mdnorm w$ and subexponential in complexity estimate for the Haar shifts, one also need a linear in $\Mdnorm w$ estimate of the so-called paraproducts. Since the paraproducts have a fixed complexity $1$, one does not care about growth of the estimates with complexity, and any (linear in $[w]\ci{A_2}$) estimate of the paraproduct, for example one obtained in \cite{LPR} would work. 

Here we would like to show how using Bellman function approach to get linear in $\Mdnorm w$ estimate of a general paraproduct from the estimate for the simplest paraproduct in $\R$ obtained by O.~Beznosova \cite{OB}. 

Let us recall the main definitions. Let a dyadic lattice $\cD$ in $\R^d$ be fixed. 

\begin{df}
Let $\f$ be a locally integrable function. A paraproduct $\Pi_\f$ with symbol $\f$ is defined by
\[
\Pi_\f f := \sum_{I\in\cD} \La f\Ra\ci I \Delta\ci I \f.
\]
\end{df}

It is well known that the paraproduct $\Pi_\f$ is bounded in  $L^2$ (unweighted) if and only if $\f\in\BMOd$, i.e.~if and only if
\[
\sup_{J\in\cD} \frac{1}{|J|} \sum_{I\in\cD: I \subset J} \| \Delta\ci I \f\|_2^2 =: \|\f\|\ci{\BMOd}^2 <\infty. 
\]
This statement is in fact equivalent to the dyadic Carleson Embedding Theorem, and from the sharp estimates in the Embedding Theorem one can get that norm of $\Pi_\f$ is bounded by $2 \|\f\|\ci{\BMOd}$. 

As for the linear in $\Mdnorm w$ weighted estimates of the paraproducts, it was proved by O.~Beznosova in \cite{OB} that for the paraproduct on the real line 
\[
\|\Pi_\f\|_{L^2(w)\to L^2(w)} \le C_1 \|f\|\ci{\BMOd} \Mdnorm w.
\]  
where $C_1$ is an absolute constant. 

This result can be rewritten as
\begin{align}
\label{eq-para-w-1}
	\sum_{I\in\cD} | \La f\Ra\ci I | \|\Delta\ci I \f\|_2 | \La g\Ra\ci{I_1} - | \La g\Ra\ci{I_1}| \cdot |I|^{1/2} \le C \Mdnorm w \|\f\|\ci{\BMOd} \|f\|\ci{L^2(w)} \|g\|\ci{L^2(w^{-1})}
\end{align}
for all $f\in L^2(w)$ and all $g\in L^2(w^{-1}$ (because changing the signs in the decomposition $\f = \sum_{I\in\cD} \Delta\ci I \f$ does not change $\|\f\|\ci{\BMOd}$). Here $\|\fdot\|_2$ is the unweighted $L^2$ norm, $I_1$ and $I_2$ denote the children of $I$ and $C=C_1/2$. 
$\mathbf b$

We want to extend this result to $\R^d$. We will show how to prove the following proposition

\begin{prop}
\label{para-A2}
Let $\f\in\BMOd(\R^d)$, and let $\Pi_\f$ be the corresponding paraproduct in $L2(\R^d)$. Let a weigh $w$  in $\R^2$ satisfies the dyadic Muckenhoupt condition $\Atd$. Then $\Pi_\f$ is bounded in $L^2(w)$ with the norm at most $162 2^{2d-2} C \|\f\|\ci{\BMOd} \Mdnorm w$, where $C$ is the constant from \eqref{eq-para-w-1}.  
\end{prop}

We present the proof of this result using the Bellman function from the result for the real line. 
Note that because of homogeneity it is sufficient to prove Proposition \ref{para-A2} only for $\|\f\|\ci\BMOd\le 1$. 

\subsection{Bellman function for paraproducts on the real line}
For an interval $I_0\in\cD$ let us fix the quantities
\begin{align}
	\label{eq-av-fg-2}
	\La f\Ra\ci{I_0} &=\bff, \qquad \La f^2 w\Ra\ci{I_0} =\bF, \qquad
	\La g\Ra\ci{I_0} =\bg, \quad \La g^2 w^{-1}\Ra\ci{I_0} =\bG
	\\
\label{eq-av-uv-2}
	\La w \Ra\ci{I_0} & =\bu, \qquad \La w^{-1} \Ra\ci{I_0} = \bv. 
	\\
\label{eq-M-Carl}
	M&= |I_0|^{-1} \sum_{I\in\cD:\,I\subset I_0} \|\Delta\ci I \f\|_2^2
\end{align}
and define the function $\cB\ci A(\bff, \bg, \bF, \bG, \bu, \bv, M)$ as
\begin{align*}
\cB\ci A(\bff, \bg, \bF, \bG, \bu, \bv, M) = |I_0|^{-1} \sup \sum_{I\in\cD: I\subset I_0} |\La f\Ra\ci I |\cdot \|\Delta\ci I \f\|_2 | \La g\Ra\ci{I_1} - | \La g\Ra\ci{I_1}| \cdot |I|^{1/2}
\end{align*}
where the supremum is taken over all $f\in L^2(w)$, $g\in L^2(w^{-1}$, all weights $w$, $\Mdnorm w\le A$ and  all $\f\in \BMOd$, $\| \f\|\ci\BMOd \le 1$, satisfying \eqref{eq-av-fg-2}--\eqref{eq-M-Carl}.

\subsubsection{Properties of Bellman function}

\begin{enumerate}
\label{s-prop-B-para}
\item Domain $\dom(\cB\ci A)$ is given by the conditions
\begin{align*}
\bu, \bv > 0, \qquad 1 \le \bu\bv \le A, \qquad \bff^2 \le \bF \bv, \quad \bg^2 \le \bG \bu, \qquad 0\le M\le 1;
\end{align*}
this simply means that for any choice of the appropriate functions the corresponding averages satisfy these constrains and that for any $7$-tuple of reals satisfying these constrains there are functions with corresponding averages over $I_0$, so the supremum is well defined (not $-\infty$).

\item Range:
\[
0\le B (X) \le C A \bF^{1/2}\bG^{1/2}. 
\]

\item The main inequality: Let $X=(\bff, \bg, \bF, \bG, \bu, \bv)$. Then for all 
\[
(X,M), (X_{1/2},\ M_{1,2})\in \dom(\cB\ci A),
\]
such that  $X=(X_1+X_2)/2$, $M-(M_1+M_2)/2 =: \bd\ge 0$, the following inequality holds:
\begin{align*}
\cB\ci A(X, M) - \left( \cB\ci A (X_1, M_1) + \cB\ci A (X_2, M_2)  \right)/2 \ge \bd |f| \cdot |\bg_1-\bg_1| .
\end{align*}
\end{enumerate}

The proof of the properties of the Bellman function is pretty standard, one can do it following the lines of Section \ref{s-prop_b}. 

\subsection{The main estimate}
The proof of Proposition \ref{para-A2} follows easily from the  lemma below. The details of the reduction are essentially the same as for the Haar shifts (slicing, remodeling on the real line and then estimating each slice), so we leave it as an easy exercise for the reader. 
\begin{lm}[The main estimate]
\label{l-main_est-para}
Let $\cB_A$, $A>1$ be a family of functions satisfying conditions \cond1, \cond2, \cond3 from Section \ref{s-prop-B-para}.  
Let $I_0$ be a dyadic interval, and let for all $I\in\ch_k(I_0) $, $0\le k\le n$, the points $X\ci I = (\bff\ci I, \bg\ci I, \bF\ci I, \bG\ci I, \bu\ci I, \bv\ci I)$, $M\ci I$, $(X\ci I \in \dom \cB\ci A$ be given. 
Assume that $X\ci I$ satisfy the dyadic martingale dynamics, namely that if $I_1$, $I_2$ are the children of $I$, then 
\begin{align*}
X\ci I = \left(X\ci{I_1} + X\ci{I_2} \right)/2. 
\end{align*}
Assume also that $M\ci I = (M\ci{I_1}+ M\ci{I_2})/2$ for all $I$ except $I_0$, and let
\[
M\ci{I_0} - 2^{-n} \sum_{I\in\ch_n(I_0)} M\ci I =: \bd\ci{I_0} \ge 0. 
\]
Then for $A'=4.5A$ we have 
\begin{align*}
\bd\ci{I_0}  
| \bff\ci{I_0}|  
&\left(2^{-n} \sum_{I\in \ch_n(I_0)} | \bg\ci I -\bg\ci{I_0} | \right)
\\ &
\le
36 \left( \cB\ci{A'} ( X\ci{I_0}, M\ci{I_0} ) - 2^{-n} \sum_{I\in \ch_n(I_0)} \cB\ci{A'} (X\ci I, M\ci{I}) \right). 
\end{align*}
%%
%where $C$ is an absolute constant. 
\end{lm}

To prove this lemma one can just literally follow the proof of Lemma \ref{l-main_est}. The main difference (a simplification) is that when we are picking weights $\alpha\ci I$, $\alpha\ci I|\le 1/3$ for the first step, we only need to get the estimate \eqref{move-g}, which can be done trivially (because we only need to control one parameter) and with the constant $1/6$ instead of $1/12$. This give the constant 36 instead of 72 in the conclusion of the lemma. 

Then we define 
\begin{align*}
(X^\pm, M^{\pm}) =  (X^\pm\ci{I_0}, M^{\pm}\ci{I_0}) = 2^{-n} \sum_{I\in\ch_n(I_0)} (X\ci I,  M\ci I), 
\end{align*}
and the rest of the proof is the same as for Lemma \ref{l-main_est} with $X^\pm\ci I$ replaced by $(X^\pm\ci I M^\pm\ci I)$. \hfill\qed 

\section{Concluding remarks}
The heart of this paper is definitely Lemma \ref{l-main_est}, which is a very general fact about convex function. Concave functions with concavity estimated by $|\sD \bff|\cdot |\sD \bg|$ are very common in the dyadic Harmonic analysis: they appear if one writes the estimates of the bilinear form of dyadic martingale multipliers $T_\sigma$. 

Lemma \ref{l-main_est} gives a simple way to transfer the estimates from the dyadic martingale multipliers to  more general martingale transforms, like dyadic shifts, etc. As the story of the $A_2$ conjecture illustrates, such transference until now was considered highly non-trivial, even in the case of simplest possible non-multiplier martingale transforms (like the simplest Haar shift considered by S.~Petermichl).

It also probably worth mentioning that for the second degree polynomials an analogue of Lemma \ref{l-main_est} was known for some time. Namely, it is a known fact (cf.~\cite{DTV-quad-2008}) that if $Q[\ldots, x, y, \ldots]$ is a quadratic form such that 
\begin{align*}
Q[\ldots, x, y, \ldots] \ge 2|xy|, 
\end{align*}
then there exists $\alpha>0$ such that 
\begin{align*}
Q[\ldots, x, y, \ldots] \ge \alpha x^2 + \alpha^{-1} y^2.
\end{align*}
Then for $\theta_k\ge 0$
\begin{align*}
\sum_k \theta_k Q[\ldots, x_k, y_k, \ldots] \ge \alpha \sum_k \theta_k x_k^2   + \alpha^{-1} \sum_k \theta_k y_k^2 \ge 
2 \left( \sum_k \theta_k x_k^2\right)^{1/2}  \left( \sum_k \theta_k y_k^2\right)^{1/2}. 
\end{align*}
Therefore, if $\Phi$ is a second degree polynomial of variables $X=( \ldots, x, y, \ldots)$ such that for $X=(X_1+X_2)/2$
\begin{align*}
\Phi(X) - (\Phi(X_1) + \Phi(X_2))/2 \ge |x_1-x_2|\cdot |y_1-y_2|
\end{align*}
then for any $\theta_k\ge 0$, $\sum_k \theta_k =1$ and for any $X$, $X_k$ such that $X=\sum_k \theta_k X_k$, 
\begin{align*}
\Phi(X) - \sum_k \theta_k \Phi(X_k) \ge 4\left( \sum_k \theta_k |x-x_k|^2\right)^{1/2} \left( \sum_k \theta_k |y-y_k|^2\right)^{1/2}
\end{align*}
If we take $2^n$ terms and put $\theta_k =2^{-n}$ we get  (for the second degree polynomials) a stronger version of Lemma \ref{l-main_est}, where we estimate above not the product of $L^1$ norms, but the bigger product of $L^2$ norms.

%\begin{align}
%	\sum_{I\in\cD} \La f\Ra\ci I \|\Delta\ci I \f\|_2 \|\Delta\ci I g\|_2 \le C \Mdnorm w \|\f\|\ci{\BMOd} \|f\|\ci{L^2(w)} \|g\|\ci{L^2(w^{-1})}
%\end{align}
%%

%\newpage

%\edexplanation

%\ednotemessage
\end{document}